\documentclass[twoside,11pt]{amsart}
\usepackage{a4wide}
\usepackage[utf8]{inputenc}
\usepackage[english]{babel}
\usepackage[fixlanguage]{babelbib}
\usepackage{color,graphicx}
\usepackage{amssymb,amsfonts,amsmath,amsthm,amscd, epsfig}
\usepackage[all]{xy}
\usepackage{tensor}
\usepackage{microtype}
\usepackage{tikz-cd}
\usepackage{float}

\title{Partial Groupoid Representations and a relation with the Birget-Rhodes expansion}
\author[Lautenschlaeger and Tamusiunas]{Wesley G. Lautenschlaeger and Thaísa Tamusiunas}
\address{Instituto de Matem\'{a}tica, Universidade Federal do Rio Grande do Sul,  Av. Bento Gon\c{c}alves, 9500, 91509-900. Porto Alegre-RS, Brazil}
\email{wesleyglautenschlaeger@gmail.com}
\email{thaisa.tamusiunas@gmail.com}
\date{}

\usepackage{amssymb}
\usepackage{amsmath}
\usepackage{amsthm}
\usepackage{indentfirst}
\usepackage{tikz}
\usetikzlibrary{cd}

\newcounter{contador}
\numberwithin{contador}{section}

\newtheorem{theorem}[contador]{Theorem}

\newtheorem{lemma}[contador]{Lemma}
\newtheorem{corollary}[contador]{Corollary}

\theoremstyle{definition}
\newtheorem{defi}[contador]{Definition}
\newtheorem{obs}[contador]{Remark}
\newtheorem{exe}[contador]{Example}

\newcommand{\G}{\mathcal{G}}
\newcommand{\cA}{\mathcal{A}}

\begin{document}

\maketitle

\begin{abstract}
We introduce partial representation of a finite groupoid $\G$ on an algebra $\cA$ and we show that the partial representations of $\G$ are in one-to-one correspondence with the representations of the algebra generated by the Birget-Rhodes expansion $\G^{BR}$ of $\G$.
\end{abstract}

\vspace{0.5 cm}



\section{Introduction}
Many studies concerning actions and partial actions of groupoids as been investigate in the last few years. For instance, the relation between partial and global actions, Galois theory, generalizations of classic theorems of group theory, Morita theory, crossed products and duality theorems were research topics addressed in \cite{bagio2020restriction}, \cite{bagio2022galois}, \cite{beier2023generalizations}, \cite{cortes2017characterisation}, \cite{della2023groupoid}, \cite{paques2018galois}, \cite{pedrotti2023injectivity}. The ideia of to classify something partial in terms of something global help us to understand the behavior of the partial theory.

The Birget-Rhodes expansion $\G^{BR}$ of an ordered groupoid $\G$ was constructed by Gilbert in \cite{gilbert2005actions} and it was proven that it has an ordered groupoid structure \cite[Proposition 3.1]{gilbert2005actions}. Also, there is a one-to-one correspondence between partial actions of $\G$ and actions of $\G^{BR}$, which can be viewed as a partial-to-global result achieved by enlarging the acting groupoid.

The construction of $\G^{BR}$ can be used to improve the work developed by Exel, Dokuchaev and Piccione in \cite{dokuchaev2000partial}, regarding to partial representations of a group $G$. In that work, the authors presented the partial group algebra of a group $G$, called $K_{par}(G)$, which is the algebra whose representations correspond to the partial representations of $G$. The algebra $K_{par}(G)$ was shown to be a groupoid algebra $K\Gamma(G)$, where $\Gamma(G)$ is a determined finite groupoid associated to $G$ \cite[Corollary 2.7]{dokuchaev2000partial}. But the groupoid $\Gamma(G)$ has a very rich structure. It is, in fact, the Birget-Rhodes expansion of the group $G$. What is behind this is the Ehresmann-Schein-Nambooripad Theorem (or simply, “ESN Theorem”) \cite[Theorem 4.1.8]{lawson1998inverse}, which proves that the category of the inverse semigroups is isomorphic to the category of inductive groupoids. Then, althought the Birget-Rhodes expansion $G^{BR}$ of $G$ is an inverse semigroup, it also has a structure of a groupoid, via the ESN Theorem.

Based on this, our major purpose in this paper is to complement and to extend the work of Exel, Dokuchaev and Piccione, establishing a relation between the partial representations of a groupoid $\G$ and representations of the algebra $K\G^{BR}$ generated by its Birget-Rhodes expansion $\G^{BR}$.  Indeed, we shall prove that there is a one-to-one correspondence between the partial representations of $\G$ and the representations of $K\G^{BR}$. This agrees with the idea of to enlarge the groupoid $\G$ to characterize partial representations in terms of \,``global" representations.

The paper is organized as follows. We start by fixing some terminology concerning groupoids and hence we introduce the concept of partial representation of a finite groupoid $\G$ on an algebra. Next we present the algebra $K_{par}(\G)$, which the representations are in one-to-one correspondence with the partial representations of $\G$. The last section is destined to prove that the algebra $K_{par}(\G)$ is the algebra generated by the Birget-Rhodes expansion of $\G$.

% In \cite{exel1998partial} Exel presented an universal inverse semigroup $S(G)$ associated to the group $G$, whose actions are shown to be in one-to-one correspondence with the partial actions of $G$. This inverse semigroup $S(G)$ was further proven by Kellendonk and Lawson in \cite{kellaw} to be the Birget-Rhodes expansion $G^{BR}$ of $G$.

% This structures have something in common: the groupoid algebra  $K_{par}(G)$ is the algebra generated by the Birget-Rhodes expansion of $G$. What is behind this is the Ehresmann-Schein-Nambooripad Theorem (hereafter, “ESN Theorem”) \cite[Theorem 4.1.8]{lawson1998inverse}, which proves that the category of the inverse semigroups and prehomomorphisms is isomorphic to the category of inductive groupoids and ordered functors. When looking at the functor of this isomorphism, the image of the inverse semigroup $S(G)$ is the groupoid $\Gamma(G)$ and vice versa.

%Applying a generalization of the ESN Theorem \cite{dewolf2018ehresmann}, it can be also viewed with an inverse semigroupoid structure. For the purposes of this paper, we shall work with the groupoid structure of $\G^{BR}$.

Throughout, rings and algebras are associative and unital.

\section{Partial representations and the algebra $K_{par}(\G)$}

\subsection{Partial representations}
We recall that a \emph{groupoid} $\G$ is a small category in which every morphism is an isomorphism. We denote by $\G_0$ the set of objects of $\G$. Observe that $\text{id}:\G_0 \rightarrow \G$, given by $\text{id}(x)=\text{id}_x$, is an injective map and whence we identify $\G_0\subset \G$. Given $g \in \G$, the \emph{domain} and the \emph{range} of $g$ will be denoted by $d(g)$ and $r(g)$, respectively. Hence, $d(g) = g^{-1}g$ and $r(g) = gg^{-1}$. For all $g, h \in G$, we write $\exists gh$ whenever the product $gh$ is defined. We fix the notation $\G_2 := \{ (g,h) \in \G \times \G : \exists gh \}$.

%If $\G$ is a groupoid we can consider the opposite groupoid $\G^{op}$ with the same elements as $\G$ but with the operation $\cdot$ given by $\exists g \cdot h$ in $\G^{op}$ if and only if $\exists hg$ in $\G$ and in this case $g \cdot h = hg$.

% A map $\varphi : \G \to \cH$ between two groupoids is said to be a \emph{groupoid homomorphism} if it is a functor. If $\varphi$ is a bijective groupoid homomorphism, then we say that $\varphi$ is a \emph{groupoid isomorphism}. In this case we write $\G \simeq \cH$.

%Let $A,B$ be unital $K$-algebras, where $K$ is a field. Recall that a $K$\emph{-algebra representation} of $A$ on $B$ is a unital $K$-algebra homomorphism $\psi : A \to B$.

For the rest of the paper, let $\G$ be a finite groupoid, $K$ be a field and $\cA$ be a $K$-algebra. We start this section by defining a partial groupoid representation on $\cA$.

\begin{defi}\label{defparrep}
A \emph{partial groupoid representation} of $\G$ on $\cA$ is a map $\pi : \G \to \cA$ such that:
\begin{enumerate}
    \item[(i)] $\pi(g)\pi(h)\pi(h^{-1}) = \pi(gh)\pi(h^{-1})$, $\forall (g,h) \in \G_2$;

    \item[(ii)] $\pi(g^{-1})\pi(g)\pi(h) = \pi(g^{-1})\pi(gh)$, $\forall (g,h) \in \G_2$;

    \item[(iii)] $\pi(g)\pi(g^{-1})\pi(g) = \pi(g)$, $\forall g \in \G$ ;
    
    \item[(iv)] $\sum_{e \in \G_0} \pi(e) = 1_A$ and $\pi(e)\pi(f) = 0$ for $e,f \in \G_0$ such that $e \neq f$.
\end{enumerate}
\end{defi}

\begin{obs}\label{obs2}
    Observe that $\pi(g) = \pi(g)\pi(g^{-1})\pi(g) = \pi(g)\pi(g^{-1}g) = \pi(g)\pi(d(g))$, for all $g \in \G$. Analogously, $\pi(g) = \pi(r(g))\pi(g)$, for all $g \in \G$. So, if $(g,h) \notin \G_2$, then $\pi(g)\pi(h) = \pi(g)\pi(d(g))\pi(r(h))\pi(h) = \pi(g)0\pi(h) = 0$.
 \end{obs}

\begin{lemma}\label{lemma33}
Let $g \in \G$. It follows that: \begin{itemize}\item[(a)] $\varepsilon(g) = \pi(g)\pi(g^{-1})$ is an idempotent of $A$; \item[(b)] if $r(g) = r(h)$, then $\varepsilon(g)\varepsilon(h) = \varepsilon(h)\varepsilon(g)$. \end{itemize}
\end{lemma}
\begin{proof}
(a) It is straightforward.
%\begin{align*}
%     \varepsilon(g)^2 & = \pi(g)\pi(g^{-1})\pi(g)\pi(g^{-1}) = \pi(g)\pi(g^{-1}) = \varepsilon(g).
% \end{align*}

(b) If $r(g) = r(h)$,
\begin{equation}\label{comut}
\begin{split}
    \pi(h^{-1})\varepsilon(g) & = \pi(h^{-1})\pi(g)\pi(g^{-1}) = \pi(h^{-1}g)\pi(g^{-1}) = \pi(h^{-1}g)\pi(g^{-1}h)\pi(h^{-1}g)\pi(g^{-1}) \\
    & = \pi(h^{-1}g)\pi(g^{-1}h)\pi(h^{-1}gg^{-1}) = \varepsilon(h^{-1}g) \pi(h^{-1}r(g)) = \varepsilon(h^{-1}g) \pi(h^{-1}r(h))\\
    & = \varepsilon(h^{-1}g)\pi(h^{-1}),
\end{split}
\end{equation}
from where it follows
\begin{align*}
    \varepsilon(h)\varepsilon(g) = \pi(h)\pi(h^{-1})\varepsilon(g) = \pi(h)\varepsilon(h^{-1}g)\pi(h^{-1}) = \varepsilon(r(h)g)\pi(h)\pi(h^{-1}) = \varepsilon(g)\varepsilon(h).
\end{align*}
\end{proof}

\subsection{The algebra $K_{par}(\G)$}

As in the case of partial representations of a group G, which can be characterized by an algebra homomorphism defined on the partial group algebra $K_{par}(G)$, in the case of partial representations of a groupoid $\G$, it is possible to construct an algebra associated to $\G$ which characterizes partial representations by algebra homomorphisms. 

\begin{defi}
 We define the partial groupoid $K$-algebra $K_{\text{par}}(\G)$ as the universal $K$-algebra with unit  $1_{K_{\text{par}}(\G)}$ generated by the set of symbols $\{[g] : g \in \G\}$ and relations:
    \begin{enumerate}
        \item[(i)] $[g^{-1}][g][h] = [g^{-1}][gh]$, $\forall (g,h) \in \G_2$;

        \item[(ii)] $[g][h][h^{-1}] = [gh][h^{-1}]$, $\forall (g,h) \in \G_2$;

        \item[(iii)] $[r(g)][g] = [g] = [g][d(g)]$, $\forall g \in \G$;
    \
        \item[(iv)] $[g][h] = 0$, $\forall (g,h) \notin \G_2$.
    \end{enumerate}
\end{defi}

Notice that $\sum_{e \in \G_0} [e] = 1_{K_{\text{par}}(\G)}$. In fact, $\left ( \sum_{e \in \G_0} [e] \right )[g] = \sum_{e \in \G_0} [e][g] = [r(g)][g] = [g]$. Similarly $[g]\left ( \sum_{e \in \G_0} [e] \right ) = [g]$. 

\begin{exe}
Let $\G = \G_1 \cup \G_2$ (disjoint union), where $\G_1 = \{g,g^{-1},r(g),d(g)\}$ and $\G_2 = \{ r(h), h \}$ with $h = h^{-1}$. Then $K_{\text{par}}(\G)$ has basis $\{[g],[g^{-1}],[r(g)],[d(g)],[g][g^{-1}],[g^{-1}][g],[h],$ $[r(h)],[h][h]\}$ as a $K$-vector space. It is easy to see that $K_{\text{par}}(\G) \simeq K_{\text{par}}(\G_1) \oplus K_{\text{par}}(\G_2)$. More generally, if $\G$ is a finite groupoid with connected components $\G_1, \ldots, \G_n$, then $K_{\text{par}}(\G) \simeq K_{\text{par}}(\G_1) \oplus \cdots \oplus K_{\text{par}}(\G_n)$.
\end{exe}

The next theorem shows that there exists a one-to-one correspondence between partial representations of $\G$ and representations of $K_{\text{par}}(\G)$.

\begin{theorem} \label{teo1}
Let $\pi : \G \to \cA$ be a partial representation of $\G$ on $\cA$. Then there exists a unique homomorphism of $K$-algebras $\phi : K_{\text{par}}(\G) \to \cA$ such that $\phi([g]) = \pi(g)$. Conversely, if $\phi : K_{\text{par}}(\G) \to \cA$ is a homomorphism of $K$-algebras, then $\pi(g) = \phi([g])$ is a partial groupoid representation of $\G$ on $\cA$.
\end{theorem}
\begin{proof}
Let $\pi : \G \to A$ be a partial representation of $\G$ on $\cA$. Define
\begin{align*}
    \phi : K_{\text{par}}(\G) & \to \cA \\
            \sum_{i=1}^m k_i \prod_{j=1}^n [g_{i,j}] & \mapsto \sum_{i=1}^m k_i\prod_{j=1}^n \pi(g_{i,j}),
\end{align*}
where $k_i \in K$ and $g_{i,j} \in \G$ for all $1 \leq i \leq m$ and $1 \leq j \leq n$. Then $\phi([g]) = \pi(g)$, for all $g \in \G$. Furthermore, if $\exists gh$, 
\begin{align*}
    \phi([g][h]) & = \phi([r(g)][g][h]) = \phi([gg^{-1}][g][h]) = \phi([g][g^{-1}][g][h]) = \phi([g][g^{-1}][gh]) \\
    & = \pi(g)\pi(g^{-1})\pi(gh) = \pi(g)\pi(g^{-1})\pi(g)\pi(h) = \pi(g)\pi(h) = \phi([g])\phi([h]),
\end{align*}
and $\phi([g][h]) = 0 = \pi(g)\pi(h)$ otherwise.

Moreover, $\phi(1_{K_{\text{par}}(\G)}) = \phi\left ( \sum_{e \in \G_0} [e] \right ) = \sum_{e \in \G_0} \pi(e) = 1_A$. Clearly, $\phi$ is unique.

Conversely, let $\phi : K_{\text{par}}(\G) \to \cA$ be a homomorphism of $K$-algebras. Define $\pi : \G \to K_{\text{par}}(\G)$ by $\pi(g) = \phi([g])$, for all $g \in \G$. We shall prove that $\pi$ is a partial groupoid representation of $\G$ on $\cA$. In fact,

(i): If $(g,h) \in \G_2$, then
\begin{align*}
    \pi(g)\pi(h)\pi(h^{-1}) & = \phi([g])\phi([h])\phi([h^{-1}]) = \phi([g][h][h^{-1}]) \\
    & = \phi([gh][h^{-1}]) = \phi([gh])\phi([h^{-1}]) = \pi(gh)\pi(h^{-1}).
\end{align*}

(ii): Analogous to (i).

(iii): $\pi(g)\pi(g^{-1})\pi(g) = \phi([g])\phi([g^{-1}])\phi([g]) = \phi([g][g^{-1}][g]) = \phi([g]) = \pi(g)$.

(iv): We have that
\begin{align*}
    \sum_{e \in \G_0} \pi(e) = \sum_{e \in \G_0} \phi([e]) = \phi \left ( \sum_{e \in \G_0} [e] \right ) = \phi(1_{K_{\text{par}}(\G)}) = 1_A
\end{align*}
and if $e,f \in \G_0$ with $e \neq f$, then $\pi(e)\pi(f) = \phi([e])\phi([f]) = \phi([e][f]) = \phi(0) = 0$.
\end{proof}

\section{The relation with the Birget-Rhodes expansion}

In this section we shall describe $K_{\text{par}}(\G)$ in terms of the Birget-Rhodes expansion of $\G$.

Define $X_g = \{h \in \G : r(h) = r(g)\}$, for all $g \in \G$. Observe that $X_g = X_{r(g)}$. Now let $S \subseteq X_e$ be a subset, for some $e \in \G_0$. We set:
\begin{align}\label{defps}
    P_S = \prod_{h \in S} \varepsilon(h) \prod_{h \in X_e \setminus S} (\pi(e) - \varepsilon(h)).
\end{align}

Using \eqref{comut} it is easy to see that $\pi(\ell)P_S = P_{\ell S}\pi(\ell)$, for all $\ell \in \G$ with $d(\ell) = e$ and $S \subseteq X_e$.

Observe that if $e \notin S$, then $P_S = 0$. Moreover, if $\ell \in X_e \setminus S$, then $P_S\pi(\ell) = 0$, because $\varepsilon(\ell)\pi(\ell) = \pi(\ell)$, so $(\pi(e) - \varepsilon(\ell))\pi(\ell) = \pi(e)\pi(\ell) - \varepsilon(\ell)\pi(\ell) = \pi(r(\ell))\pi(\ell) - \pi(\ell) = \pi(\ell) - \pi(\ell) = 0$.

Furthermore, \begin{align}\label{ps}\pi(e) = \sum_{S \subseteq X_e} P_S,\end{align} since we have the combinatorial formula
\begin{align*}
    \pi(e) & = \prod_{h \in X_e} \pi(e) = \prod_{h \in X_e} (\pi(e) - \varepsilon(h) + \varepsilon(h)) \\ & = \sum_{S \subseteq X_e} \left ( \left ( \prod_{h \in S} \varepsilon(h) \right ) \cdot \left ( \prod_{h \in X_g \setminus S} (\pi(e) - \varepsilon(h))  \right ) \right ).
\end{align*}

Now define $Y_g = \{ h \in \G : r(h) = d(g) \} = X_{g^{-1}}$. We set the finite groupoid, constructed from $\G$, $$\G^{BR} = \{ (A,g) : d(g), g^{-1} \in A \subseteq Y_g\}$$ as the groupoid with partial multiplication given by
    \begin{align*}
        (A,g) \cdot (B,h) = \begin{cases}
            (B,gh), \text{ if } (g,h) \in \G_2 \text{ and } A = hB, \\
            \text{undefined, otherwise.}
        \end{cases}
    \end{align*}

The inverse of the pair $(A,g)$ is $(gA,g^{-1})$. Also $d(A,g) = (A,d(g))$ and $r(A,g) = (gA,r(g))$. This groupoid is the Birget-Rhodes expansion of $\G$ (see \cite[Proposition 3.1]{gilbert2005actions}).

% Consider $\G^{BR}$, the Birget-Rhodes expansion of $\G$, defined in \cite{gilbert2005actions} as the groupoid
% \begin{align*}
%     \G^{BR} = \{(U,g) : g,r(g) \in U \subseteq X_g \}
% \end{align*}
% with product
% \begin{align*}
%     (U,g)(V,h) = \begin{cases}
%         (U,gh), \text{ if } \exists gh \text{ and } U = gV, \\
%         \text{undefined, otherwise.}
%     \end{cases}
% \end{align*}

% %Then we obtain the following result.

% \begin{prop} \label{propbrop}
%     The groupoids $\Gamma(\G)$ and $(\G^{BR})^{op}$ are isomorphic.
% \end{prop}
% \begin{proof}
%     Consider $\varphi : \Gamma(\G) \to \G^{BR}$ defined by $\varphi(A,g) = (A,g^{-1})$. Notice that this is well-defined, since $(A,g) \in \Gamma(\G)$ implies that $d(g),g^{-1} \in A \subseteq Y_g$. Now, $r(g^{-1}) = d(g)$, from where it follows that $r(g^{-1}),g^{-1} \in A \subseteq X_{g^{-1}}$, that is, $(A,g^{-1}) \in \G^{BR}$.

%     Now, $\exists (A,g)(B,h) \text{ in } \Gamma(\G) \Leftrightarrow (g,h) \in \G_2 \text{ and } A = Bh \Leftrightarrow (h^{-1},g^{-1}) \in \G_2 \text{ and } B = Ah^{-1} \Leftrightarrow \exists (B,h^{-1})(A,g^{-1})$ in $\G^{BR}$. In this case,
%     \begin{align*}
%         \varphi((A,g)(B,h)) = \varphi((B,gh)) = (B,h^{-1}g^{-1}) = (B,h^{-1})(A,g^{-1}) = \varphi(B,h)\varphi(A,g).
%     \end{align*}

%     Therefore if we obtain a homomorphism of groupoids $\phi : \Gamma(\G) \to \G^{BR}$. It is easy to see that $\varphi$ is a bijection, from where the result follows.
% \end{proof}

An easy calculation shows that the elements of the form $(A,e)$, $e \in \G_0$, are idempotents in the groupoid algebra $K\G^{BR}$, that is, $(A,e)^2 = (A,e)$. Also, they are mutually orthogonal and their sum is $1_{K\G^{BR}}$.

%\sum_{A \ni g^{-1}}

\begin{lemma}
Define the map $\lambda : \G \to K\G^{BR}$ by $\lambda(g) = \sum_{A \ni g^{-1} } (A,g)$. Then $\lambda$ is a partial groupoid representation of $\G$ on $K\G^{BR}$.  
\end{lemma}
\begin{proof} (i) Given $(g,h) \in \G_2$, we have
\begin{align*}
    \lambda(g^{-1})\lambda(g)\lambda(h) & = \sum_{\substack{A \ni g \\ C \ni g^{-1} \\ B \ni h^{-1}}} (A,g^{-1})(B,g)(C,h) = \sum_{\substack{C \ni h^{-1} \\ hC \ni g^{-1}}} (ghC,g^{-1})(hC,g)(C,h) \\
    & = \sum_{\substack{C \ni h^{-1} \\ C \ni h^{-1}g^{-1}}} (C,d(g)h) = \sum_{\substack{C \ni h^{-1} \\ C \ni h^{-1}g^{-1}}} (C,r(h)h) = \sum_{\substack{C \ni h^{-1} \\ C \ni h^{-1}g^{-1}}} (C,h).
\end{align*}

On the other hand,

\begin{align*}
    \lambda(g^{-1})\lambda(gh) & = \sum_{\substack{A \ni g \\ C \ni h^{-1}g^{-1}}} (A, g^{-1})(C,gh) = \sum_{\substack{C \ni h^{-1} \\ C \ni h^{-1}g^{-1}}} (hC, g^{-1})(C,gh) = \sum_{\substack{C \ni h^{-1} \\ C \ni h^{-1}g^{-1}}} (C,h).
\end{align*}

Hence $\lambda(g^{-1})\lambda(g)\lambda(h) = \lambda(g^{-1})\lambda(gh)$, for all $(g,h) \in \G_2$. The equality $\lambda(g)\lambda(h)\lambda(h^{-1}) = \lambda(gh)\lambda(h^{-1})$ is proved similarly.

(iii) The equality $\lambda(g)\lambda(g^{-1})\lambda(g) = \lambda(g)$ is equivalent to  $\lambda(g)\lambda(d(g)) = \lambda(g)$ by Remark \ref{obs2}. We shall show the second equality. For $g \in \G$,
\begin{align*}
    \lambda(g)\lambda(d(g)) & = \left ( \sum_{A \ni g} (A,g) \right) \left( \sum_{B \ni d(g)} (B,d(g)) \right ) = \sum_{A \ni g} (A,g)(A,d(g)) = \sum_{A \ni g}(A,g) = \lambda(g).
\end{align*}

(iv) We have $\sum_{e \in \G_0} \lambda(e) = \sum_{e \in \G_0} \sum_{A \ni e} (A,e) = 1_{K\G^{BR}}$, and if $e,f \in \G_0$, $e \neq f$, \[\lambda(e)\lambda(f) = \left (\sum_{A \ni e} (A,e) \right ) \left ( \sum_{B \ni f} (B,f) \right ) = \sum_{\substack{A \ni e \\ B \ni f}} (A,e)(B,f) = 0. \qedhere\]
\end{proof}

\begin{theorem} \label{teorep}
There is a one-to-one correspondence between the partial groupoid representations of $\G$ and the representations of $K\G^{BR}$. More precisely, if $\cA$ is any unital $K$-algebra, then $\pi : \G \to \cA$ is a partial groupoid representation of $\G$ if and only if there is an algebra homomorphism $\tilde{\pi} : K\G^{BR} \to \cA$ such that $\pi = \tilde{\pi} \circ \lambda$. Moreover, such a homomorphism $\tilde{\pi}$ is unique.
\end{theorem}
\begin{proof}
    If $\tilde{\pi} : K\G^{BR} \to \cA$ is a homomorphism of $K$-algebras, then clearly $\pi = \tilde{\pi} \circ \lambda : \G \to \cA$ is a partial groupoid representation of $\G$ on $\cA$.

    Conversely, assume that $\pi : \G \to \cA$ is a partial groupoid representation of $\G$. For all $g \in \G$, denote by $\varepsilon(g) = \pi(g)\pi(g^{-1}) \in \cA$. Recall from Lemma \ref{lemma33} and Remark \ref{obs2} that $\varepsilon(g)\varepsilon(h) = \delta_{r(g),r(h)}\varepsilon(h)\varepsilon(g)$. Also, from \eqref{comut}, if $d(g)=r(h)$ then $\pi(g)\varepsilon(h) = \varepsilon(gh)\pi(g)$, and if $d(g) \neq r(h)$, then $\pi(g)\varepsilon(h) = 0$. Similarly, if $r(g) = r(h)$, then $\varepsilon(h)\pi(g) = \pi(g)\varepsilon(g^{-1}h)$, and $\varepsilon(h)\pi(g) = 0$ otherwise. 

    For $(A,g) \in \G^{BR}$, we define
    \begin{align*}
        \tilde{\pi}(A,g) = \pi(g) \left ( \prod_{h \in A} \varepsilon(h) \right ) \left ( \prod_{h \in Y_g \setminus A} (\pi(d(g)) - \varepsilon(h)) \right ).
    \end{align*}

    For $(A,g), (B,h) \in \G^{BR}$, we have
    \begin{align*}
        \tilde{\pi}(A,g)\tilde{\pi}(B,h) & = \pi(g) \cdot \prod_{k \in A} \varepsilon(k) \cdot \prod_{k \in Y_g \setminus A} (\pi(d(g)) - \varepsilon(k)) \cdot \pi(h)\\ & \cdot \prod_{\ell \in B} \varepsilon(\ell) \cdot \prod_{\ell \in Y_h \setminus B} (\pi(d(h)) - \varepsilon(\ell)).
    \end{align*}

    If $d(g) \neq r(h)$, then $\pi(d(g))\pi(h) = 0$ and $\varepsilon(k)\pi(h) = 0$, for all $k \in Y_g \setminus A$, since $\varepsilon(k) = \pi(k)\pi(k^{-1})$ and $d(k^{-1}) = r(k) = d(g) \neq r(h)$. So in this case $\tilde{\pi}(A,g)\tilde{\pi}(B,h) = 0$.

    Suppose now that $d(g) = r(h)$. Hence
    \begin{align*}
        & \tilde{\pi}(A,g)\tilde{\pi}(B,h) \\
        & = \pi(g)\pi(h) \cdot \prod_{k \in A} \varepsilon(h^{-1}k) \cdot \prod_{k \in Y_g \setminus A} (\pi(d(h)) - \varepsilon(h^{-1}k)) \cdot \prod_{\ell \in B} \varepsilon(\ell) \cdot \prod_{\ell \in Y_h \setminus B} (\pi(d(h)) - \varepsilon(\ell)) \\
        & = \pi(g)\pi(h) \cdot \prod_{k \in h^{-1}A} \varepsilon(k) \cdot \prod_{k \in Y_h \setminus h^{-1}A} (\pi(d(h)) - \varepsilon(k)) \cdot \prod_{\ell \in B} \varepsilon(\ell) \cdot \prod_{\ell \in Y_h \setminus B} (\pi(d(h)) - \varepsilon(\ell)).
    \end{align*}
    %where ($*$) follows from the observations in the beginning of the proof and from the fact that $\pi(d(g))\pi(h) = \pi(r(h))\pi(h) = \pi(h) = \pi(h)\pi(d(h))$.

    If $h^{-1}A \neq B$, that is, if $A \neq hB$, then either there is $k \in h^{-1}A$ such that $k \in Y_h \setminus B$ or there is $k \in B$ such that $k \in Y_h \setminus h^{-1}A$. In both cases, the factor $\varepsilon(k)(\pi(d(h)) - \varepsilon(k)) = 0$ appears in the expression of $\tilde{\pi}(A,g)\tilde{\pi}(B,h)$, from where it follows that $\tilde{\pi}(A,g)\tilde{\pi}(B,h) = 0$.

    On the other hand, if $h^{-1}A = B$, then
    \begin{align*}
        \tilde{\pi}(A,g)\tilde{\pi}(B,h) & = \pi(g)\pi(h) \cdot \prod_{k \in h^{-1}A} \varepsilon(k) \cdot \prod_{k \in Y_h \setminus h^{-1}A} (\pi(d(h))) - \varepsilon(k) \\
        & = \pi(g)\pi(h)\varepsilon(h^{-1}) \cdot \prod_{\substack{k \in h^{-1}A \\ k \neq h^{-1}}} \varepsilon(k) \cdot \prod_{k \in Y_h \setminus h^{-1}A} (\pi(d(h))) - \varepsilon(k) \\
        & = \pi(gh)\varepsilon(h^{-1}) \cdot \prod_{\substack{k \in h^{-1}A \\ k \neq h^{-1}}} \varepsilon(k) \cdot \prod_{k \in Y_h \setminus h^{-1}A} (\pi(d(h))) - \varepsilon(k) \\
        & = \pi(gh) \cdot \prod_{k \in h^{-1}A} \varepsilon(k) \cdot \prod_{k \in Y_h \setminus h^{-1}A} (\pi(d(h))) - \varepsilon(k) \\
        & = \tilde{\pi}(h^{-1}A,gh) = \tilde{\pi}(B,gh) = \tilde{\pi}((A,g) \cdot (B,h)).
    \end{align*}

Therefore, in all cases, we have $\tilde{\pi}(A,g)\tilde{\pi}(B,h) = \tilde{\pi}((A,g) \cdot (B,h))$. This shows that extending $\tilde{\pi}$ linearly from $\G^{BR}$ to $K\G^{BR}$ we obtain a homomorphism of $K\G^{BR}$ in $\cA$.

    Now recall that $1_{K\G^{BR}} = \sum_{e \in \G_0} \sum_{A \ni e} (A,e)$. Then
    \begin{align*}
        \tilde{\pi}(1_{K\G^{BR}}) & = \tilde{\pi}\left ( \sum_{e \in \G_0} \sum_{A \ni e} (A,e) \right ) = \sum_{e \in \G_0} \sum_{A \ni e} \tilde{\pi}(A,e) \\
        & = \sum_{e \in \G_0} \sum_{A \ni e} \pi(e) \cdot \prod_{A \ni g} \varepsilon(g) \cdot \prod_{g \in Y_e \setminus A} (\pi(e) - \varepsilon(g)) \\
        & = \sum_{e \in \G_0} \sum_{A \ni e} \prod_{A \ni g} \varepsilon(g) \cdot \prod_{g \in Y_e \setminus A} (\pi(e) - \varepsilon(g)) \\
        & \overset{\eqref{defps}}{=} \sum_{e \in \G_0} \sum_{A \ni e} P_A = \sum_{e \in \G_0} \sum_{A \subseteq X_e} P_A \overset{\eqref{ps}}{=} \sum_{e \in \G_0} \pi(e) = 1_{\cA}.
    \end{align*}

    Moreover,
    \begin{align*}
        \tilde{\pi} \circ \lambda(g) & = \tilde{\pi} \left ( \sum_{A \ni g^{-1}} (A,g) \right ) = \sum_{A \ni g^{-1}} \tilde{\pi}(A,g) = \pi(g) \cdot  \sum_{A \ni g^{-1}} \sum_{h \in A} \varepsilon(h) \cdot \prod_{h \in Y_g \setminus A}(\pi(d(g)) - \varepsilon(h))
        \end{align*}
        \begin{align*}
        & = \pi(g)\varepsilon(g^{-1})\cdot  \sum_{A \ni g^{-1}} \sum_{\substack{h \in A \\ h \neq g^{-1}}} \varepsilon(h) \cdot \prod_{h \in Y_g \setminus A}(\pi(d(g)) - \varepsilon(h)) \\
        & = \pi(g) \cdot  \sum_{A \ni g^{-1}} \sum_{\substack{h \in A \\ h \neq g^{-1}}} \varepsilon(h) \cdot \prod_{h \in Y_g \setminus A}(\pi(d(g)) - \varepsilon(h)) \\
        & = \pi(g) \cdot  \sum_{A \ni g^{-1}} \pi(d(g)) \sum_{\substack{h \in A \\ h \neq g^{-1}}} \varepsilon(h) \cdot \prod_{h \in Y_g \setminus A}(\pi(d(g)) - \varepsilon(h)) \\
        & = \pi(g) \cdot  \sum_{A \ni g^{-1}} (\varepsilon(g^{-1}) + \pi(d(g)) - \varepsilon(g^{-1})) \sum_{\substack{h \in A \\ h \neq g^{-1}}} \varepsilon(h) \cdot \prod_{h \in Y_g \setminus A}(\pi(d(g)) - \varepsilon(h)) \\
        & = \pi(g) \cdot \sum_{B}\prod_{h \in B}\varepsilon(h) \cdot \prod_{h \in Y_g \setminus B} (\pi(d(g)) - \varepsilon(h)) \\
        & = \pi(g)\tilde{\pi}\left (\sum_{B}(B,d(g)) \right ) = \pi(g)\tilde{\pi}(d(g)) = \pi(g) \cdot \left ( \sum_{e \in \G_0} \tilde{\pi}(e) \right ) = \pi(g)1_{K\G^{BR}} = \pi(g).
    \end{align*}

    Now it only remains to show the uniqueness of the homomorphism $\tilde{\pi}$. To this claim, we shall show that $\lambda(\G)$ generates $K\G^{BR}$.

    Let $(B,h) \in \G^{BR}$, where $B = \{b_1^{-1}, b_2^{-1}, \ldots, b_{k-1}^{-1}, h^{-1}\}$ is a subset of $Y_h$ containing the $d(h)$. The set of such pairs forms a vector space basis for $K\G^{BR}$. Let us denote by $\mathfrak{A}$ the subalgebra of $K\G^{BR}$ generated by $\lambda(\G)$. Let $\{g_1, \ldots, g_k\} \subseteq \G$ be such that
    \begin{align*}
        g_1 = b_1, \quad g_1g_2 = b_2, \quad g_1g_2g_3 = b_3, \quad \ldots \quad g_{1}g_{2} \cdots g_{k-1} & = b_{k-1}, \quad g_1 \cdots g_k = h.
    \end{align*}

    These elements are well-defined. In fact, from $g_1 = b_1$, $d(b_1^{-1}) = r(b_1) = d(h) = r(b_2)$ and $g_1g_2 = b_2$ one obtain $g_2 = b_{1}^{-1}b_2$. Inductively, we obtain $g_i = b_{i-1}^{-1}b_i$, for all $2 \leq i \leq k-1$ and $g_k = b_{k-1}^{-1}h$. 

    Consider the element
    \begin{align*}
        \lambda(g_1) \cdots \lambda(g_k) & = \sum_{\substack{A_1 \ni g_1^{-1} \\ \vdots \\ A_k \ni g_k^{-1}}} (A_1,g_1) \cdots (A_k,g_k) = \sum_{\substack{A_k \ni g_k^{-1} \\ g_{k}A_k \ni g_{k-1}^{-1} \\ \vdots \\ g_2 \cdots g_kA_k \ni g_1^{-1}}} (A_1, g_1 \cdots g_k) = \sum_{A \supseteq B} (A,h).
    \end{align*}

    Thus, for all $(B,h) \in \G^{BR}$, $\sum_{A \supseteq B} (A,h) \in \mathfrak{A}$. Suppose that $Y_g \setminus B = \{ x_1, x_2, \ldots, x_n \}$. We have that
    \begin{align*}
        \sum_{A \supseteq B} (A,h) - \sum_{A \supseteq B \cup \{x_1 \}} (A,h) = \sum_{\substack{A \supseteq B \\ A \not\ni x_1}} (A,h),
    \end{align*}
    from where it follows inductively that
    \begin{align*}
        (B,h) = \sum_{\substack{A \supseteq B \\ A \not\ni x_1, \ldots, x_{n-1}}} (A,h) - \sum_{\substack{A \supseteq B \\ A \not\ni x_1, \ldots, x_{n}}} (A,h) \in \mathfrak{A},
    \end{align*}
    ending the proof.
\end{proof}

\begin{corollary} \label{coliso}
The groupoid algebra $K\G^{BR}$ is isomorphic to the partial groupoid algebra $K_{\text{\emph{par}}}(\G)$.
\end{corollary}
\begin{proof}
    The maps $[ \,\,] : \G \to K_{\text{par}}(\G)$, $g \mapsto [g]$, and $\lambda : \G \to K\G^{BR}$ are partial representations of $\G$. By Theorems \ref{teo1} and \ref{teorep}, there exist $K$-algebra homomorphisms $\tilde{\pi} : K\G^{BR} \to K_{\text{par}}(\G)$ and $\phi : K_{\text{par}}(\G) \to K\G^{BR}$ such that $\tilde{\pi}(\lambda(g)) = [g]$ and $\phi([g]) = \lambda(g)$. 

    It is easy to check that $\tilde{\pi}$ and $\phi$ are each others inverses, since they are the identity on the generators; hence $\tilde{\pi}$ and $\phi$ are isomorphisms. %The result is now clear since $\Gamma(\G) \simeq (\G^{BR})^{op}$ by Proposition \ref{propbrop}.
\end{proof}

\begin{obs}
The structure of $K_{\text{par}}(\G)$ regarding to Bernoulli partial actions was extensively studied in \cite{velasco2021algebras}, and several examples were presented.
\end{obs}

\renewcommand{\bibname}{Bibliography}

\end{document}